\documentclass[12pt]{article}
\usepackage{amssymb,amsmath,amsthm,latexsym,graphicx}

\newtheorem{theorem}{Theorem}[section]
\newtheorem{lemma}[theorem]{Lemma}

\newtheorem{corollary}[theorem]{Corollary}

\newtheorem{claim}[theorem]{Claim}

\newtheorem{question}[theorem]{Question}

\newcommand{\fami}{\mathcal{F}}  	

\title{
Improved bounds on the chromatic numbers of the square of Kneser graphs}

\author{
\begin{tabular}{c}
{\sc Seog-Jin KIM\thanks{This work was supported by the National Research Foundation of Korea(NRF) grant funded by the Korea government (MEST) (No. 2011-0009729).}} \\
[1ex]
{\small
Department of Mathematics Education}\\
\small{Konkuk University,
Seoul 143-701, Korea} \\
{\small
{\it E-mail address}: {\tt skim12@konkuk.ac.kr}} \\
\\
{\sc Boram PARK\thanks{Corresponding author: borampark22@gmail.com}}\\
[1ex]
{\small
The School of Computational Science}\\
{\small Korea Institute for Advanced Study, Seoul 130-722, Korea} \\
{\small
{\it E-mail address}: {\tt borampark22@gmail.com}} \\
\end{tabular}
}

\begin{document}

\maketitle

\begin{abstract}
The Kneser graph $K(n,k)$ is the graph whose vertices are the $k$-elements subsets of an $n$-element set, with two vertices adjacent if the sets are disjoint. The square $G^2$ of  a graph $G$ is the graph defined on $V(G)$ such that two vertices $u$ and $v$ are adjacent in $G^2$ if the distance between $u$ and $v$ in $G$ is at most 2.
Determining the chromatic number of the square of the Kneser graph $K(2k+1, k)$ is an interesting
problem, but not much progress has been made.
Kim and Nakprasit~\cite{2004KN}  showed that $\chi(K^2(2k+1,k)) \leq 4k+2$, and
Chen, Lih, and Wu ~\cite{2009CLW} showed  that $\chi(K^2(2k+1,k)) \leq 3k+2$ for  $k \geq 3$.
In this paper, we give improved upper bounds on $\chi(K^2(2k+1,k))$.
We show that $\chi(K^2(2k+1,k)) \leq 2k+2$, if $ 2k +1 = 2^n -1$ for some  positive integer $n$.
Also we show that  $\chi(K^2(2k+1,k)) \leq \frac{8}{3}k+\frac{20}{3}$ for every integer $k\ge 2$.
In addition to giving improved upper bounds, our proof is
concise and can be easily understood by readers while the proof
in \cite{2009CLW} is very complicated.
Moreover, we show that $\chi(K^2(2k+r,k))=\Theta(k^r)$ for each integer $2 \leq r \leq k-2$.
\end{abstract}







\section{Introduction}

For a finite set $X$, let ${X \choose k}$ be the set of all $k$-element subsets of $X$.
For $n \geq 2k$, for a finite set $X$ with $n$-elements,
the {\em Kneser graph} $K(n,k)$ is the graph whose vertex set is
${X \choose k}$ and two vertices $A$ and $B$ adjacent if and only if $A\cap B =\emptyset$.

Kneser graphs have many interesting properties and have been the subject of many researches. The problem of computing the chromatic number of a Kneser graph was conjectured by Kneser  and proved by Lov\'{a}sz~\cite{1978L} that $\chi(K(n,k))=n-2k+2$.
Also, several types of colorings of Kneser graphs have been considered.


For a simple graph $G$, the {\it square} $G^2$ of $G$ is defined such that $V(G^2) = V(G)$ and two vertices $x$ and $y$ are adjacent in $G^2$ if and only if the distance between $x$ and $y$ in $G$ is at most 2.
We denote the square of the Kneser graph $K(n,k)$ by $K^2(n,k)$.
The problem of computing $\chi (K^2(n,k))$, which was originally posed by F\"{u}redi, was introduced and discussed in \cite{2004KN}. As an independent set of $K^2(n,k)$
Note that that $A$ and $B$ are adjacent in $K^2(n,k)$ if and only if $A\cap B=\emptyset$ or $|A\cap B|\ge 3k-n$.
Therefore, $K^2(n,k)$ is the complete graph $K_t$ where $t = {n \choose k}$ if $n\ge 3k-1$, and  $K^2(n,k)$ is a perfect matching if $n=2k$.
But for  $2k+1\le n \le 3k-2$, the exact value of $\chi(K^2(n,k))$ is not known.
Hence it is an interesting problem to determine the chromatic number of the square of the Kneser graph $K(2k+1, k)$ as the first nontrivial case.
In 2004, Kim and Nakprasit~\cite{2004KN} showed that $\chi(K^2(2k+1,k))\le 4k$ if $k$ is odd and
 $\chi(K^2(2k+1,k))\le 4k+2$ if $k$ is even.
And then, in 2009, Chen, Lih, and Wu ~\cite{2009CLW} improved the bound as $\chi(K^2(2k+1,k))\le 3k+2$ for $k\ge 3$.

In this paper, we give improved  upper bounds on $\chi(K^2(2k+1,k))$ with a concise proof.
We show that $\chi(K^2(2k+1,k)) \leq \frac{8}{3}k+\frac{20}{3}$ for any integer  $k\ge 2$.
In particular, when $2k+1 = 2^n -1$ for some positive integer $n$, we give better upper bounds.
We show that $\chi(K^2(2k+1,k)) \leq 2k +2$, if $2k+1 = 2^n -1$ for some positive integer $n$.
Note that the proof
in \cite{2009CLW} is very complicated.
However,
our proof in this paper is
concise and can be easily understood by readers.

Considering the problem determining $\chi(K(n, k))$, which was solved by Lov\'{a}sz~\cite{1978L}, we can expect that the problem determining $\chi(K^2(2k+1, k))$ is difficult.
There is not even any conjecture on the  value of $\chi(K^2(2k+1, k))$.
We have observed that $\chi(K^2(2k+1,k)) \leq \alpha k + \beta$ for some real numbers $\alpha$ and $\beta$.  A natural interesting problem is determine  the least value of $\alpha$.
From the result in \cite{2004KN} and the result that $\chi(K^2(2k+1,k)) \leq \frac{8}{3}k+\frac{20}{3}$, we know that $1 \leq \alpha \leq  \frac{8}{3}$.

On the other hand, from the result that $\chi(K^2(2k+1,k)) \leq 2k +2$ for infinitely many special cases,
we can conjecture that
$\chi(K^2(2k+1,k)) \leq 2k +2$ for all  integers $k \geq k_0$ for some fixed integer $k_0$.
Regarding this conjecture, we give a supporting evidence.  We show that
for any fixed real number $\epsilon > 0$,
\[
\limsup_{k \rightarrow \infty}  \frac{\chi(K^2 (2k+1, k))}{k} \leq 2 + \epsilon.
\]

\medskip

In addition, we study $\chi(K^2(2k+r,k))$ where $2 \leq r \leq k-2$.
For $2 \leq r \leq k-2$,
we show that  ${k+r \choose r} + 1 \leq \chi(K^2 (2k+r,k)) \leq (r+2)(3k + \frac{3r+3}{2})^r$.  These are the first results for $2 \leq r \leq k-2$.
From the well-known fact that $(\frac{a}{b})^b < {a \choose b}$ for any $b < a$,
we can conclude that
for each integer $1 \leq r \leq k-2$, there exist constant real numbers $\alpha_r$ and $\beta_r$ such that
$\alpha_r k^r \leq \chi(K^2(2k+r,k)) \leq \beta_r k^r$, where $\alpha_r$ and $\beta_r$ depend on $r$.
That is,
$\chi(K^2(2k+r,k))=\Theta(k^r)$ for each integer $1 \leq r \leq k-2$.


\medskip

The coloring of the square of Kneser graphs is closely related with an intersecting family as follows.
A family ${\fami}$ is called an {\it $(n, k, L)$-system}
if ${\fami} \subset {[n] \choose k}$ and $|A \cap B| \in L$  for all
distinct elements $A$ and $B$ of ${\fami}$.
Let $m(n, k, L)$ denote the maximum size of an $(n, k, L)$-system.  The problem of determining
$m(n, k, L)$ was introduced by Deza, Erd\H{o}s, and Frankl in~\cite{De_Er_PF}. (See~\cite{PF-Ota} for more results about $m(n, k, L)$.)
The coloring of the square of Kneser graphs is related with $(n, k, L)$-system, since $m(2k+1, k, L) =  \alpha(K^2 (2k+1, k))$ when $L=
\left\{1, 2, \cdots, k-2 \right\}$.
Thus good upper bounds on $m(2k+1, k, L)$ provide good lower bounds on $\chi(K^2(2k+1, k))$
from the inequality that $\frac{n(H)}{\alpha(H)} \leq \chi(H)$ for every graph $H$.
Kim and Nakprasit \cite{2004KN} showed that $\chi(K^2 (9, 4)) \geq 11$ by showing that
$m(9, 4, \left \{1, 2\right\}) = 12$.  Later, Khodkar and Leach \cite{2009KL} showed that
$\chi(K^2(9, 4)) \leq 11$.  Thus it is concluded that $ \chi(G^2) = 11 = \big\lceil \frac{n(G^2)}{\alpha(G^2)}\big\rceil$ where $G = K(9, 4)$.

\medskip

In addition, the coloring of the square of a graph $G$ can be explained in a relation with $L(p, q)$-labelling problem introduced by Griggs and Yeh \cite{Griggs1992}.  For  nonnegative integers $p$ and $q$, an $L(p, q)$-labelling of a graph $G$ is a function $\phi : V(G) \rightarrow \{0, 1, \ldots, k \}$ such that
$|\phi(u) - \phi(v)| \geq p$ if $d_G (u, v) = 1$ and $|\phi(u) - \phi(v)| \geq q$ if $d_G (u, v) = 2$.
The $L(p,q)$-labelling number of $G$, denoted by $\lambda_{p,q}(G)$, is the least integer $k$ such that $G$ admits an
$L(p, q)$-labelling $\phi : V(G) \rightarrow \{0, 1, \ldots, k \}$.
The $L(p, q)$-labelling problem has attracted a considerable amount of interest.
Note that $\chi(G^2) = \lambda_{1, 1}(G) +1$.
(See \cite{Calamoneri} for a survey.)

\section{Coloring of the Kneser graph $K^2(2k+1,k)$}

Let $H$ be a finite additive abelian group and let $2^H$ be the set of subsets of $H$.
We define a function $\sigma_{H}: 2^H \rightarrow H$ such that for any subset $X$ of $H$, $\sigma_H (X)=\sum_{x\in X} x$, where the sum is over the addition of the group $H$. If there is no confusion, we denote $\sigma_H$ by $\sigma$.
The following two lemmas are simple, but will play a key role in the proofs of main results.

\begin{lemma}\label{lem:sum}
If $A$ and $B$ are two subsets of an abelian group $H$ such that $|A| = |B| = k$ and $|A \cap B| = k-1$ for some positive integer $k$, then
$\sigma_H(A)\neq \sigma_H(B)$.
\end{lemma}

\begin{proof}
Let $A\setminus B=\{g\}$ and $A\setminus B=\{h\}$ where $g\neq h$.
From the definition of $\sigma_H$, $\sigma_H(A)=\sigma_H(A\cap B) +g$ and  $\sigma_H(B)=\sigma_H(A\cap B) +h$.  Hence $\sigma_H(A)\neq \sigma_H(B)$ since $g \neq h$.
\end{proof}

For a positive integer $m$, let $\mathbb{Z}_m=\{ 0,1,\ldots, m-1 \}$ be the cyclic group of order $m$, and let $\mathbb{Z}^*_{m}=\mathbb{Z}_m\setminus\{ 0 \}$.
Let $\mathbb{Z}_2^n$ denote the direct product of $n$ copies of $\mathbb{Z}_2$, and let $\mathbb{Z}_2^n \times \mathbb{Z}_q$ denote the direct product of $\mathbb{Z}_2^n$ and $\mathbb{Z}_q$.

Let $\prec$ be the lexicographic ordering of $\mathbb{Z}_{2}^n \times\mathbb{Z}_{q}^*$ by considering $\mathbb{Z}_2$ and $\mathbb{Z}^*_q$ as subsets of $\mathbb{N}$.
That is, $(x_1, \ldots,x_{n+1}) \prec (y_1,\ldots,y_{n+1})$ if and only if there exists a positive integer $m$ in $\{1, \ldots, n+1\}$ such that
$x_i = y_i$ for all $i < m$ and $x_m < y_m$.
For convenience, let $F=\mathbb{Z}_{2}^n \times\mathbb{Z}_{q}^*$.
For any positive integer $\ell$, let $F_{\ell} $ be the set of the first $\ell$ elements of $F$ in the lexicographic ordering $\prec$.
For example, when $\ell=q-1$, the first $q-1$ elements of $F$ are $(0,\ldots,0,1)$, $(0,\ldots,0,2)$, \ldots, $(0,\ldots,0,q-1)$. That is,
$F_{q-1}=\{ (0,\ldots,0,1), (0,\ldots,0,2), \ldots, (0,\ldots,0,q-1) \}$.

\begin{lemma}\label{lem:even}
For an odd integer $q$ and for an integer $n$ with $n \geq 2$, let $F=\mathbb{Z}_{2}^n \times\mathbb{Z}^*_{q}$.
For a positive even integer $\ell$, let
\[X=\mathbb{Z}_{2}^n \times\mathbb{Z}_{q}\setminus( F_{\ell}\cup \{ (0,0,\ldots,0)\}).\]
Then the followings holds.
\begin{itemize}
\item[(i)] If $\ell \ge q-1$, then every element $(x_1,\ldots,x_{n+1})$ in $X$ has a coordinate $x_i$ such that
$x_i\neq 0$ and $1 \leq i \leq n$.
\item[(ii)]For each $1\le i\le n$, the number of elements in $X$ whose $i$th entry is 1 is even.
\end{itemize}
\end{lemma}

\begin{proof}
Since the first $q-1$ elements of $F=\mathbb{Z}_2^n\times \mathbb{Z}^*_{q}$ are $(0,\ldots,0,1)$, $(0,\ldots,0,2)$, \ldots, $(0,\ldots,0,q-1)$,
\[  \{ (x_1, \ldots,x_{n},x_{n+1})\in X : x_i= 0 \mbox{ for all } 1 \leq i \leq n \} \subseteq F_{\ell}\cup \{ (0,\ldots,0) \}.\]
Thus by the definition of $X$, each element $(x_1,\ldots,x_{n+1})$ in $X$ has a coordinate $x_i$ such that
$x_i\neq 0$ and $1 \leq i \leq n$.
Thus (i) holds.

\medskip

Next, for each $i$ $(1 \leq i \leq n)$, let $X_i$ denote the set of elements in $X$ whose $i$th entry is 1.  Then
\[\begin{array}{l}
|X_i|
=
| \{\mathbf{u} \in \mathbb{Z}_{2}^n \times\mathbb{Z}_{q} : i\text{th entry of } \mathbf{u} \text{ is 1}\}| - | \{\mathbf{w} \in F_{\ell} : i\text{th entry of } \mathbf{w} \text{ is 1}\}|
\end{array}.\]
Note that $| \{\mathbf{u} \in \mathbb{Z}_{2}^n \times\mathbb{Z}_{q} : i\text{th entry of } \mathbf{u} \text{ is 1}\}| = 2^{n-1}q$ where $n \geq 2$.
Hence to prove $(ii)$, it is enough to show that $| \{\mathbf{w} \in F_{\ell} : i\text{th entry of } \mathbf{w} \text{ is 1}\}|$ is even.

For each $1 \leq j \leq \frac{\ell}{2}$,  let $P_{j}$ be the set of the $(2j-1)$th element  and the $(2j)$th element in $F$ in the ordering $\prec$.
Then, since $\ell$ is even, $F_{\ell}$ is the disjoint union of $P_1$, $P_2$,\ldots, $P_{\frac{\ell}{2}}$.
Thus
\[| \{\mathbf{w} \in F_{\ell} : i\text{th entry of } \mathbf{w} \text{ is 1}\}| =\sum_{1\le j\le \frac{\ell}{2}} | \{\mathbf{w} \in P_j : i\text{th entry of } \mathbf{w} \text{ is 1}\}|.\]
Let $\mathbf{a}=(a_1,\ldots,a_n,a_{n+1})$ and $\mathbf{b}=(b_1,\ldots,b_n,b_{n+1})$ be the $(2j-1)$th element and $(2j)$th element in $F$.
Since $(2j-1)$ is odd,  $a_{n+1}$ is odd. Therefore $a_{n+1}<q-1$ since $q-1$ is even.
By the definition of $\prec$, it has to be $\mathbf{b}=(a_1,\ldots,a_n,a_{n+1}+1)$.
Therefore $a_i=1$ if and only if $b_i=1$, which implies that $|\{\mathbf{w} \in P_j : i\text{th entry of } \mathbf{w} \text{ is 1}\}|$ is zero or two for each $1\le j \le \frac{\ell}{2}$.
Thus  $| \{\mathbf{w} \in F_{\ell} : i\text{th entry of } \mathbf{w} \text{ is 1}\}|$ is even, and consequently $|X_i|$ is  even. Hence (ii) holds.
\end{proof}

Now we will prove the main result.

\begin{theorem} \label{result(2k+1)}
For integers $k, n, p$, and  $r$,
if $2k+1=(2^n-1)p+r$ where  $p\ge 1$, $n\ge2$, and  $0\le r\le 2^n-2$, then
\[\chi(K^2(2k+1, k))\le
\left\{\begin{array}{lll}
2^np&= 2k+1+p,  & \text{ if }r=0 \\
2^n(p+1)&=2k+1+p-r+2^n, & \text{ if }r \text{ is odd} \\
2^n(p+2)&=2k+1+p-r+2^{n+1},  &  \text{ if }r \text{ is even} \mbox{ and } r\neq 0.  \end{array}\right.\]
\end{theorem}

\begin{proof} For each case of the following, we will define a group $H$ and a subset $X$ of $H$ such that $|X|=2k+1$ to define the Kneser graph $K(2k+1, k)$ on the ground set $X$.

\medskip

\noindent\textbf{Case 1.} $r=0$.

\noindent  Let $H = \mathbb{Z}_2^n \times \mathbb{Z}_p$ be the group obtained by the direct product of $\mathbb{Z}_2^n$ and $\mathbb{Z}_p$.
Let $X = \mathbb{Z}_2 ^n \times  \mathbb{Z}_p \setminus\{ (0,\ldots,0,i): i\in \mathbb{Z}_p \}$.
Then $|X|=p 2^n-p=2k+1$.

\medskip

\noindent\textbf{Case 2.} $r$ is an odd integer.

\noindent Since $2k+1$ and $r$ are odd, $p$ is even and so $p+1$ is odd.
In this case, let $H$ be the group $\mathbb{Z}_2 ^n \times  \mathbb{Z}_{p+1}$.
Let $\ell=2^n-1-r+p$ and $F=\mathbb{Z}_{2}^n \times\mathbb{Z}^*_{p+1}$.  We define $X$ as follows.
\[X=\mathbb{Z}_{2}^n \times\mathbb{Z}_{p+1}\setminus( F_{\ell}\cup \{ (0, \ldots, 0) \}).\]
Note that $\ell=2^n-1-r+p$ is an even integer, and  $X$ is a subset of $H$ such that $|X|=2^n(p+1)-(2^n-1-r+p+1)=2^np-p+r=2k+1$.

\medskip

\noindent\textbf{Case 3.}  $r$ is an even integer, $r\neq 0$.

\noindent In this case, let $H = \mathbb{Z}_2 ^n \times  \mathbb{Z}_{p+2}$.
Since $2k+1$ is odd and  $r$ is even, $p$ is odd and so $p+2$ is odd.
Let $\ell=2^{n+1}-2-r+(p+1)$ and $F=\mathbb{Z}_{2}^n \times\mathbb{Z}^*_{p+2}$.  We define $X$ as follows.
\[X=\mathbb{Z}_{2}^n \times\mathbb{Z}_{p+2}\setminus( F_{\ell} \cup \{ (0,\ldots,0)\}).\]
Note that $\ell=2^{n+1}-2-r+(p+1)$ is an even integer, and $X$ is a subset of $H$ such that $|X|=2^n(p+2)-(2^{n+1}-2-r+p+1+1)=2^np-p+r=2k+1$.

\bigskip

First, we will show that the following claim.

\bigskip

\begin{claim} \label{claim-one}
For any element $(x_1,\ldots,x_{n+1})\in X$, $x_i\neq 0$ for some $i\in \{1,2,\ldots,n\}$.
\end{claim}

\medskip

\begin{proof}
If $r=0$, then it is clear from the definition of $X$.
Suppose that $r\neq 0$.
Since $0 \le r \le 2^n-2$,
\[ \ell = \left\{ \begin{array}{lll}
2^n-1-r+p & \ge p  & \text{ if }r \mbox{ is odd} \\
2^{n+1}-2-r+(p+1)& \ge p+1 & \text{ if } r \mbox{ is even}. \\
\end{array}
\right.\]
Moreover, if $r$ is odd,  then $F=\mathbb{Z}_{2}^n \times\mathbb{Z}^*_{p+1}$ and $p+1$ is odd. And
if $r$ is even, then $F=\mathbb{Z}_{2}^n \times\mathbb{Z}^*_{p+2}$ and $p+2$ is odd.
Therefore  $x_i\neq 0$ for some $i\in \{1,2,\ldots,n\}$ by  (i) of  Lemma~\ref{lem:even}.
This completes the proof of Claim \ref{claim-one}.
\end{proof}

\bigskip

Let $G$ be the Kneser graph defined on the set $X$, that is,
$V(G) = {X \choose k}$.
Define $\sigma_H : V(G^2) \rightarrow H$, where $\sigma_H (X)=\sum_{x\in X} x$.  We will show that $\sigma_H(A) \neq \sigma_H (B)$ for each edge $AB$ of $G^2$.
We denote $\sigma_H$ by $\sigma$ for simplicity.

Let $AB$ be an edge in $G^2$.
Note that $|A \cap B| = 0$ or $|A \cap B| = k-1$.
If $|A \cap B| = k-1$, then $\sigma(A)\neq \sigma(B)$ by Lemma~\ref{lem:sum}.
Hence it is remained to show that  $\sigma(A)\neq \sigma(B)$ when $|A \cap B| = 0$.

From now on, we assume that $|A \cap B| = 0$.
Since $|A|=|B|=k$ and $|X|=2k+1$, there exists unique element $\mathbf{z}\in X$ such that $\mathbf{z}\not\in A\cup B$.
Let $\mathbf{z}=(z_1,\ldots, z_{n+1})$.
By Claim~\ref{claim-one},  $z_i\neq 0$ for  some $i\in \{1,2,\ldots,n\}$.  Fix an $i$ such that $z_i \neq 0$ and $i\in \{1,2,\ldots,n\}$.
Let $X_i=\{ (a_1,\ldots,a_{n+1})\in  X :  a_i=1 \}$, which is the set of elements in $X$ whose $i$th entry is 1.
Now we can show that $|X_i|$ is even as follows.
For the case when $r=0$, we have that $|X_i|=|\mathbb{Z}_2^{n-1}\times \mathbb{Z}_p|=2^{n-1}\times p$.
And for the case when $r\neq 0$, since $\ell$ is an even integer,  $|X_i|$ is even
by (ii) of Lemma~\ref{lem:even}.

\begin{claim}\label{claim:entry}
For any subset $A$ of $X$, $|A\cap X_i| \pmod{2}$ is the $i$th entry of $\sigma(A)$ where $1 \leq i \leq n$.
\end{claim}

\begin{proof}
Let $A=\{ \mathbf{a}_1,\mathbf{a}_2,\ldots,\mathbf{a}_k \}\subset X$.
For each $1\le j\le k$,   denote $\mathbf{a}_j =(a_{j1},a_{j2}, \ldots,a_{j(n+1)})$.
Then
\begin{eqnarray*}
\sigma(A) &=&\mathbf{a}_1+\mathbf{a}_2+\cdots+\mathbf{a}_k = ( \sum_{j=1}^{k}a_{j1},  \sum_{j=1}^{k}a_{j2}, \ldots,  \sum_{j=1}^{k}a_{j(n+1)}).
\end{eqnarray*}

For $1 \leq i \leq n$, since $a_{ji}$ is 0 or 1,
\[\sum_{j=1}^{k}a_{ji} = \left\{\begin{array}{ll}  0 &\text{ if there are even number of } \mathbf{a}_j \text{ in } A
\text{  such that } a_{ji}=1  \\
 1 &\text{ if there are odd number of } \mathbf{a}_j \text{ in } A \text{ such that } a_{ji}=1.  \end{array}\right. \]
Note that $|A\cap X_i|$ is the number of elements of $A$ whose $i$th entry is 1.
Therefore $|A\cap X_i| \pmod{2}$ is the $i$th entry of $\sigma(A)$.  This completes the proof of Claim \ref{claim:entry}.
\end{proof}

Note that $\mathbf{z}\in X_i$ by the definition of $\mathbf{z}$.
Since $X_i\setminus\{\mathbf{z} \}$ is the disjoint union of $A\cap X_i$ and $B\cap X_i$,
we have that $|A\cap X_i| +|B\cap X_i|=|X_i|-1$.
As $|X_i|$ is even, $|X_i|-1$ is odd.
Therefore, one of $|A\cap X_i|$ and $|B\cap X_i|$ is odd and the other is even.
Note that by Claim~\ref{claim:entry},
$|A\cap X_i| \pmod{2}$ is the $i$th entry of $\sigma(A)$, and
$|B\cap X_i| \pmod{2}$ is the $i$th entry of $\sigma(B)$.
Thus the $i$th entry of two vectors $\sigma(A)$ and $\sigma(B)$ are distinct, and consequently $\sigma(A) \neq \sigma(B)$.

Hence $\sigma$ provides a proper coloring of $G^2$ with $|H|$ colors, where $|H|$ is the number of elements of $H$.  Therefore $\chi(G^2)\le |H|$.
\end{proof}

\medskip
Theorem~\ref{result(2k+1)} gives the following corollary.

\begin{corollary} \label{upper;2k+2}
If $2k+1 =2^n-1$ for some integer $n\ge 2$, then $\chi(K^2(2k+1, k))\le 2k+2$.
\end{corollary}
\begin{proof}
If $2k+1=2^{n+1}-1$, then  $p=1$ and $r=0$.  Thus $\chi(K^2(2k+1, k))\le 2k+2$ by Theorem \ref{result(2k+1)}.
\end{proof}

In addition, from $2k+1=(2^n-1)p+r$, we have the following relations.

\[
\begin{array}{lllll}
(a) & 2^np&= 2k+1+p &=\frac{2^{n+1}}{2^n-1}k  +\frac{2^n}{2^{n}-1} & \text{ if }r=0 \\
(b) & 2^n(p+1)&=2k+1+p-r+2^n&=\frac{2^{n+1}}{2^n-1}k + \frac{2^n(2^n-r)}{2^n-1} & \text{ if }r \text{ is odd}\\
(c) & 2^n(p+2)&=2k+1+p-r+2^{n+1}&=\frac{2^{n+1}}{2^n-1}k + \frac{2^n(2^{n+1}-r-1)}{2^n-1}  &  \text{ if }r \text{ is even}, r\neq 0  \end{array}
\]
Note that $\frac{2^n(2^n-r)}{2^n-1} \leq \frac{2^n(2^{n+1}-r-1)}{2^n-1}$ for all $n$ and $r$.
Thus
\begin{equation} \label{upper-large-n}
\chi(K^2 (2k+1, k)) \leq
\frac{2^{n+1}}{2^n-1}k + \frac{2^{n}(2^{n+1}-3)}{2^n-1}  \leq
\frac{2^{n+1}}{2^n-1} \big(k +  (2^n-1) \big)
\end{equation}

\medskip

\begin{corollary} \label{new-general}
For any integer $k\ge 2$,
$\chi(K^2(2k+1,k))\le \frac{8}{3}k+\frac{20}{3}$.
\end{corollary}
\begin{proof}
When $n =2$, from (\ref{upper-large-n}), we have
\[\chi(K^2(2k+1,k))\leq \frac{2^{n+1}}{2^n-1}k + \frac{2^{n}(2^{n+1}-3)}{2^n-1} = \frac{8}{3}k+\frac{20}{3}.\]
\end{proof}

It was showed in \cite{2009CLW} that $\chi(K^2(2k+1,k))\le 3k+2$.
Corollary \ref{new-general} improves the result in \cite{2009CLW} when $k \geq 15$.
The following corollary implies that the upper bounds on $\chi(K^2 (2k+1, k))$ become much smaller when $k$ is sufficiently large.

\begin{corollary} \label{large k}
For any fixed real number $\epsilon > 0$, there exists a positive integer $k_0$ depending on $\epsilon$ such that
\[
\chi(K^2 (2k+1, k)) \leq (2 + \epsilon)(k +  \sqrt{2k+1}),
\]
for any positive integer $k \geq k_0$.  Thus for any fixed real number $\epsilon > 0$,
\[
\limsup_{k \rightarrow \infty} \frac{\chi(K^2 (2k+1, k))}{k} \leq 2 + \epsilon.
\]
\end{corollary}
\begin{proof}
For any fixed $\epsilon > 0$, there exists $n_0$ such that $\frac{2^{n_0+1}}{2^{n_0}-1} <  2 + \epsilon$.
Fix $n_0$.
Then there exists a positive integer $k_0$ such that  $2^{n_0}-1 \leq \sqrt{2k+1}$ for all integers $k \geq k_0$.  Thus for any $k \geq k_0$, we can find integers $p$ and $r$ such that $2k+1 = (2^{n_0}-1)p + r$ and $0\le r\le 2^{n_0}-2$ as in the proof of Theorem \ref{result(2k+1)}.  Then by (\ref{upper-large-n}), it holds that
\[
\chi(K^2 (2k+1, k)) \leq
\frac{2^{n_0 +1}}{2^{n_0} -1} \big(k +  2^{n_0}-1 \big) \leq (2 + \epsilon)(k + \sqrt{2k+1}).
\]
\end{proof}


\section{Coloring of the Kneser graph $K(2k+r,k)$ for $2 \leq r \leq k-2$}

In this section, we give upper bounds on $\chi(K^2(2k+r,k))$ for all $2\le r \le k-2$.

\begin{theorem} \label{general-r}
Let $\mathbb{F}$ be a finite field with characteristic $p$ such that  $p > r$ and $|\mathbb{F}|>2k+r$.
Then $\chi(K^2(2k+r,k)) \le (r+2)|\mathbb{F}|^r$, where $|\mathbb{F}|$ is the number of elements of $\mathbb{F}$.
\end{theorem}

\begin{proof}
Let $X\subset \mathbb{F}$ with $|X|=2k+r$.
Let $G=K(2k+r,k)$ be the Kneser graph defined on the set $X$, that is,
$V(G) = {X \choose k}$.  We define a function $f:{X \choose k} \rightarrow \mathbb{F}^r
\times \mathbb{N}$ by
\[  f(A)=(\sigma_1(A),\sigma_2(A), \ldots, \sigma_r(A), \phi(A))  \]
where $\sigma_j(A) = \sum_{a\in A} a^j$ for $1 \leq j \leq r$, and
a proper ($r+2$)-coloring $\phi$ of $K(2k+r,k)$.  Note that $\chi(K(2k+r,k)) = r+2$ and so the range of $f$ has the size at most $(r+2)|\mathbb{F}|^r $.

\medskip
We will show that $f(A)\neq f(B)$  for each edge $AB$ of $G^2$.
If $AB$ is an edge of $G^2$, then $|A \cap B| \in \{0, k-r,k-r-1,\ldots, k-1 \}$.
First, if $AB$ is an edge of $G^2$ with $A\cap B=\emptyset$, then $AB$ is an edge of $G$ and so $\phi(A) \neq \phi(B)$ as $\phi$ is a proper ($r+2$)-coloring of $K(2k+r,k)$.
Hence $f(A) \neq f(B)$.

\medskip
Next, we will show that $f(A) \neq f(B)$ when $AB$ is an edge of $G^2$ with $A \cap B \neq \emptyset$.
Suppose that  $AB$ is an edge of $G^2$ with $A \cap B \neq \emptyset$ and $f(A)= f(B)$.
Then $k-r \le|A\cap B|\le k-1$.
Let $|A\cap B|=k-s$ for some $1\le  s \le r$.
We can denote $A=\{ a_1, \ldots,a_{s},x_1,\ldots,x_{k-s}\}$ and $B=\{ b_1, \ldots,b_{s},x_1,\ldots, x_{k-s}\}$ where $a_i \neq b_j$ for all $i, j \in \{1, \ldots, s\}$.
For convenience  let $A'=\{ a_1, \ldots,a_{s}\}$ and $B'=\{ b_1, \ldots,b_{s}\}$.
Note that $A'$ and $B'$ are disjoint.
Since $f(A) = f(B)$, we have that $\sigma_i(A)=\sigma_i(B)$ for each $1\le i\le s$.  Thus
\begin{eqnarray*}
a_1+a_2+\cdots+a_s&=&b_1+b_2+\cdots +b_s \\
a_1^2+a_2^2+\cdots+a_s^2&=&b_1^2+b_2^2+\cdots +b_s^2 \\
&\vdots&  \\
a_1^s+a_2^s+\cdots+a_s^s&=&b_1^s+b_2^s+\cdots +b_s^s.
\end{eqnarray*}

Let $e_i(Z)$ is the $i$th the elementary symmetric function of $Z$ that
is the sum of all distinct products of $i$ distinct elements in $Z$.
That is, when $Z=\{ z_1,\ldots,z_n\}$, ${\displaystyle e_i(Z)=\sum_{1 \leq {t_1}<{t_2}<\cdots<t_i \leq n} z_{t_1}z_{t_2}\cdots z_{t_i} }$.
For example, when $Z = \{x_1, x_2, x_3, x_4 \}$,  \ $e_1 (Z) = x_1 + x_2 + x_3 + x_4$, $e_2(Z)=x_1x_2+x_1x_3+x_1x_4+x_2x_3+x_2x_4+x_3x_4$,
$ e_3 (Z) = x_1x_2x_3 + x_1x_2x_4 + x_1x_3x_4 + x_2x_3x_4$,
and $e_4 (Z) =  x_1x_2x_3x_4$.

\medskip

Note that the well-known Newton's identity \cite{Newton2003} is stated like below.   For each $1\le i \le n$,
\[
ie_{i}(Z)= \sum_{t=1}^{i} (-1)^{t-1}e_{i-t}(Z)\sigma_t(Z),
\]

\begin{claim} \label{equal-newton}
$ e_i(A') =  e_i (B')$ for all $1\le i \le s$
\end{claim}
\begin{proof}
We can show that $ e_i(A') =  e_i (B')$ for all $1\le i \le s$ recursively.
Clearly $ e_1(A') =  e_1 (B')$. Now suppose that $e_{j}(A ') = e_{j}(B ')$ for all $1 \leq j \leq i-1$.
Next, we will show that $ e_i(A') =  e_i (B')$.
Since $\sigma_i(A') = \sigma_i(B')$ for all $1 \leq i \leq s$, we have that
\[
ie_{i}(A ')=\sum_{t=1}^{i} (-1)^{t-1}e_{i-t}(A')\sigma_t(A')= \sum_{t=1}^{i} (-1)^{t-1}e_{i-t}(B')\sigma_t(B') = ie_i(B').
\]
Thus $i e_i(A')=i e_i(B')$, and then we conclude that
$e_i(A')=e_i(B')$ since the characteristic of $\mathbb{F}$ is greater than $s$.
Thus $ e_i(A') =  e_i (B')$ for all $1\le i \le s$.  This completes the proof Claim \ref{equal-newton}.
\end{proof}
Therefore by using Claim \ref{equal-newton}, we have that
\begin{eqnarray*}
(b_1-a_1)(b_1-a_2)\cdots (b_1-a_s) &=&b_1^s-e_1(A')b_1^{s-1}+e_{2}(A')b_1^{s-2}+\cdots+(-1)^se_s(A') \\
&=&b_1^s-e_1(B')b_1^{s-1}+e_{2}(B')b_1^{s-2}+\cdots+(-1)^se_s(B')\\
&=&(b_1-b_1)(b_1-b_2)\cdots (b_1-b_s) =0.
\end{eqnarray*}
Hence $(b_1-a_1)(b_1-a_2)\cdots (b_1-a_s) = 0$.  Thus
there is an element in $a_i\in A'$ such that $b_1= a_i$, since all computations are defined in the field $\mathbb{F}$.
This is a contradiction to the fact that $A'$ and $B'$ are disjoint.  Therefore, $\sigma_i(A) \neq \sigma_i(B)$
for some $i \in \{1, \ldots, r\}$, which implies that $f(A) \neq f(B)$.
Hence $f$ gives a proper coloring of $G^2$, and $\chi(G^2)$ is at most the size of the range of $f$.
\end{proof}

The following Theorem is known, which is a generalization of Bertrand-Chebyshev Theorem.

\begin{theorem} \label{prime-2n} {\cite{prime}}
For any positive integer $ n \geq 2$, there  is a prime number $p$ such that $2n < p < 3n$.
\end{theorem}

We have the following corollary from Theorem \ref{prime-2n}.

\begin{corollary} \label{large r}
 ${k+r \choose r} + 1 \leq \chi(K^2 (2k+r,k)) \leq (r+2)(3k + \frac{3r+3}{2})^r$.
\end{corollary}
\begin{proof}
From Theorem \ref{prime-2n}, there is a prime $p$ such that $ 2k+r \leq 2(k+ \lceil \frac{r}{2} \rceil) < p < 3(k+ \lceil \frac{r}{2} \rceil) \leq 3k + \frac{3r+3}{2}$.
Thus for every positive integer $k$, there exists a field $\mathbb{F}$ such that $|\mathbb{F}| = p$ for some odd prime $p$ and
$ 2k+r < |\mathbb{F}| < 3k + \frac{3r+3}{2}$.  Thus $\chi(K^2 (2k+r,k)) \leq (r+2)(3k + \frac{3r+3}{2})^r$
from Theorem \ref{general-r}.

On the other hand, the maximum degree of $(K^2 (2k+r,k))$ is  ${k+r \choose r}$.  Thus
${k+r \choose r} +1 \leq  \chi(K^2 (2k+r,k))$.
\end{proof}


\section{Remark}
In Corollary \ref{upper;2k+2}, it was showed that
$\chi(K^2(2k+1,k)) \leq 2k +2$ for infinitely many special cases.
Note  that $\chi(K^2(7,3))=6$ and $\chi(K^2(9,4))=11$.
Hence it is an interesting  to answer the following
question.

\begin{question}
Is it true that  $\chi(K^2(2k+1,k)) \leq 2k +2$ for all  integers $k \geq k_0$ for some fixed integer $k_0$?
\end{question}




\end{document}